\documentclass[12pt,a4paper,twoside,reqno]{amsart}
\title[Gravitating vortices, cosmic strings, K\"ahler--Yang--Mills]{Gravitating vortices, cosmic strings,\\ and the K\"ahler--Yang--Mills equations}

\author[L. \'Alvarez-C\'onsul]{Luis \'Alvarez-C\'onsul}
  \address{Instituto de Ciencias Matem\'aticas (CSIC-UAM-UC3M-UCM)\\ Nicol\'as Cabrera 13--15, Cantoblanco\\ 28049 Madrid, Spain}
\email{l.alvarez-consul@icmat.es}

\author[M. Garcia-Fernandez]{Mario Garcia-Fernandez}
\email{mario.garcia@icmat.es}

\author[O. Garc\'{\i}a-Prada]{Oscar Garc\'{\i}a-Prada}
\email{oscar.garcia-prada@icmat.es}

\thanks{Partially supported by the Spanish MINECO under the ICMAT Severo Ochoa grant No. SEV-2011-0087, and under grant No. MTM2013-43963-P. The work of the second author has been partially supported by the Nigel Hitchin Laboratory under the ICMAT Severo Ochoa grant. The research leading to these results has received funding from the European Union's Horizon 2020 Programme (H2020-MSCA-IF-2014) under grant agreement No. 655162, and by the European Commission Marie Curie IRSES MODULI Programme PIRSES-GA-2013-612534}

\usepackage{hyperref,url}
\usepackage{amssymb,latexsym}
\usepackage{amsmath,amsthm}
\usepackage[latin1]{inputenc}
\usepackage{enumerate}
\usepackage[all]{xy} \CompileMatrices
\SelectTips{cm}{12} 

\usepackage{verbatim} 

\usepackage{wrapfig}
\usepackage{graphicx}

\theoremstyle{plain}
\newtheorem{theorem}{Theorem}[section]
\newtheorem{lemma}[theorem]{Lemma}

\newtheorem{proposition}[theorem]{Proposition}

\theoremstyle{definition}
\newtheorem{definition}[theorem]{Definition}
\newtheorem{definition-theorem}[theorem]{Definition-Theorem}

\newtheorem*{acknowledgements}{Acknowledgements}
\theoremstyle{remark}
\newtheorem{remark}[theorem]{Remark}

\numberwithin{equation}{section} 

\newcommand{\tr}{\operatorname{tr}}

\newcommand{\pr}{p}

\newcommand{\Id}{\operatorname{Id}}

\newcommand{\Hom}{\operatorname{Hom}}

\newcommand{\ad}{\operatorname{ad}}

\newcommand{\Aut}{\operatorname{Aut}}

\newcommand{\dbar}{\bar{\partial}}

\newcommand{\CC}{{\mathbb C}}

\newcommand{\PP}{{\mathbb P}}

\newcommand{\RR}{{\mathbb R}}

\renewcommand{\(}{\left(}

\renewcommand{\)}{\right)}

\newcommand{\vol}{\operatorname{vol}}

\newcommand{\Vol}{\operatorname{Vol}}

\newcommand{\defeq}{\mathrel{\mathop:}=} 

\newcommand{\surj}{\to\kern-1.8ex\to}

\newcommand{\lto}{\longrightarrow}

\newcommand{\lra}[1]{\stackrel{#1}{\longrightarrow}}

\newcommand{\cA}{\mathcal{A}}

\newcommand{\cJ}{\mathcal{J}}

\newcommand{\cJi}{\mathcal{J}^{i}}

\newcommand{\cP}{\mathcal{P}}

\newcommand{\cG}{\mathcal{G}}

\newcommand{\cO}{\mathcal{O}}

\newcommand{\cT}{{\mathcal{T}}}

\newcommand{\Lie}{\operatorname{Lie}}

\newcommand{\LieG}{\operatorname{Lie} \cG}

\newcommand{\cX}{{\widetilde{\mathcal{G}}}}

\newcommand{\LieX}{\operatorname{Lie} \cX}

\newcommand{\cH}{\mathcal{H}} 

\newcommand{\LieH}{\Lie\cH}

\newcommand{\SL}{\operatorname{SL}}

\newcommand{\U}{\operatorname{U}}

\newcommand{\SU}{\operatorname{SU}}

\newcommand{\PU}{\operatorname{PU}}

\begin{document}

\begin{abstract}
In this paper we construct new solutions of the K\"ahler--Yang--Mills  equations, by applying dimensional reduction methods to the product of the  complex projective line with a compact Riemann surface. The resulting  equations, that we call gravitating vortex equations, describe Abelian vortices on the Riemann surface with back 
reaction of the metric. As a particular case of these gravitating vortices on the Riemann sphere we find solutions of the Einstein--Bogomol'nyi equations, which physically correspond to Nielsen--Olesen cosmic strings in the Bogomol'nyi phase. We use this to provide a Geometric Invariant Theory interpretation of an existence result by Y. Yang for the Einstein--Bogomol'nyi equations, applying a criterion due to G. Sz\'ekelyhidi. 
\end{abstract}

\maketitle

\setlength{\parskip}{5pt}
\setlength{\parindent}{0pt}

\tableofcontents

\section{Introduction}\label{sec:intro}


The K\"ahler--Yang--Mills equations, introduced in  \cite{AGG,GF}, emerge from a natural extension of the theories for constant scalar curvature K\"ahler metrics and hermite--Yang--Mills connections. Fix a holomorphic vector bundle $E$ over a compact complex K\"ahlerian manifold $X$. The K\"ahler--Yang--Mills equations intertwine the scalar curvature $S$ of a K\"ahler metric $g_X$ on $X$ and the curvature $F$ of a hermitian metric $H$ on $E$:
\begin{equation}\label{eq:CKYM0}
\begin{split}
i\Lambda F & = \lambda \Id,\\
S - \alpha \Lambda^2 \tr F \wedge F & = c.
\end{split}
\end{equation}
Here, $\Lambda F$ is the contraction of the curvature $F$ with the K\"ahler form of $g_X$. The equations depend on a coupling constant $\alpha\in\RR$, and the constants $\lambda, c\in\RR$ are topological (see Section \ref{sec:background} for details). These equations can be defined more generally, with $E$ replaced by a holomorphic principal bundle, but this set-up will suffice for our purposes in this paper.

The initial motivation for the present work was to find the simplest non-trivial solutions of the K\"ahler--Yang--Mills equations. It turns out that these equations decouple on a compact Riemann surface, due to the vanishing of the first Pontryagin term $\tr F \wedge F$, and so in this case the solution to the problem reduces to  a combination of the uniformization theorem for Riemann surfaces and the theorem of Narasimhan and Seshadri~\cite{D2,NS} (see~\cite[Example~5.6]{AGG}). For an arbitrary higher-dimensional manifold $X$, determining whether~\eqref{eq:CKYM0} admits solutions is a difficult problem, because in this case these equations are a system of coupled fourth-order fully non-linear partial differential equations. 

Despite this, a large class of examples was found in~\cite{AGG,GF} for small $\alpha$, by perturbing constant scalar curvature K\"ahler metrics and hermite--Yang--Mills connections. More concrete and interesting solutions \textemdash{} over a polarised threefold that does not admit any constant scalar curvature K\"ahler metric \textemdash{} were obtained by Keller and T{\o}nnesen--Friedman \cite{KellerTonnesen}. The second author and Tipler~\cite{GFT} added new examples to this short list, by simultaneous deformation of the complex structures of $X$ and $E$. 

To provide a new class of interesting examples, inspired by~\cite{G1}, in this paper we consider the dimensional reduction of the K\"ahler--Yang--Mills equations from
\[
X = \Sigma \times \PP^1
\]
to a compact connected Riemann surface $\Sigma$ of arbitrary genus $g(\Sigma)$. Let $\SU(2)$ act on $X$, trivially on $\Sigma$, and in the standard way on $\PP^1 \cong \SU(2)/\U(1)$. Let $L$ be a holomorphic line bundle over $\Sigma$, and $\phi$ a holomorphic global section of $L$. Let $E$ be the $\SU(2)$-equivariant  holomorphic  rank 2 vector bundle over $X$, fitting in the holomorphic extension
\[
0\to L\lto E\lto\mathcal{O}_{\mathbb{P}^1}(2)\to 0
\]
determined by the section $\phi\in H^0(L)\cong H^1(L\otimes\mathcal{O}_{\mathbb{P}^1}(-2))$ (we omit the obvious pull-backs). For $\tau \in \RR_{>0}$, consider the $\SU(2)$-invariant K\"ahler metric on $X$ whose K\"ahler form is
\begin{equation}\label{eq:omegatau}
\omega_\tau = \omega_\Sigma + \frac{4}{\tau} \omega_{FS},
\end{equation}
where $\omega_\Sigma$ is a K\"ahler form on $\Sigma$ and $\omega_{FS}$ is the Fubini--Study metric on $\PP^1$. Our main existence result for the K\"ahler--Yang--Mills equations is the following.

\begin{theorem}\label{th:ExistenceKYM}
Suppose $\phi \neq 0$ and that
\begin{equation}\label{eq:ineqtau}
0 < c_1(L) < \frac{\tau \Vol(\Sigma)}{4\pi}.
\end{equation}
\begin{enumerate}
\item[\textup{(1)}] Assume $g(\Sigma)  \geqslant 1$. Then, there exists $\epsilon > 0$ such that for all $\alpha \in \RR$ with $0 < |\alpha| < \epsilon$ there exists a solution of the K\"ahler--Yang--Mills equations on $(X,E)$ with K\"ahler class $[\omega_\tau]$.

\item[\textup{(2)}] Assume $g(\Sigma) = 0$ and $\alpha \tau c_1(L)=1$. Let $D=\sum n_jp_j$ be the effective divisor on $\PP^1$ corresponding to the pair $(L,\phi)$. Then, the K\"ahler--Yang--Mills equations on $(X,E)$ have solutions with K\"ahler class $[\omega_\tau]$, provided that the divisor $D$ is GIT polystable for the canonical linearized $\textup{SL}(2,\CC)$-action on the space of effective divisors.
\end{enumerate}
\end{theorem}

The proof of this theorem is carried out in several steps via the analysis of the equations obtained by dimensional reduction of the K\"ahler--Yang--Mills equations, that we call gravitating vortex equations, given by
\begin{equation}\label{eq:gravvortexeq}
\begin{split}
i\Lambda F + \frac{1}{2}(|\phi|^2-\tau) & = 0,\\
S + \alpha(\Delta + \tau) (|\phi|^2 -\tau) & = c.
\end{split}
\end{equation}
The unknowns of the gravitating vortex equations are a K\"ahler metric $g_\Sigma$ on $\Sigma$ and a hermitian metric $h$ on $L$. Here, $F$ is the curvature of the Chern connection of $h$, $\Lambda F$ is its contraction by the K\"ahler form of $g_\Sigma$, $|\phi|$ is the pointwise norm of $\phi$ with respect to $h$, $S$ is the scalar curvature of $g_\Sigma$, and $\Delta$ is the Laplacian of the metric on the surface acting on functions. The constant $c\in\RR$ is topological, as it can be obtained by integrating~\eqref{eq:gravvortexeq} over $\Sigma$. More explicitly, 
\begin{equation}\label{eq:constantcintro}
c = \frac{2\pi(\chi(\Sigma) - 2\alpha\tau c_1(L))}{\Vol_\omega(\Sigma)}.
\end{equation}
The first equation in \eqref{eq:gravvortexeq} is the vortex equation, also known as the Bogomol'nyi equation in the abelian-Higgs model. Solutions of \eqref{eq:gravvortexeq} are called vortices, and have been extensively studied in the literature after the seminal work of Jaffe and Taubes~\cite{Jaffe-Taubes,Taubes1} on the Euclidean plane. It is known~\cite{Brad,G1,G2,Noguchi} that the existence of solutions (with $\phi\neq 0$) is equivalent to the inequality \eqref{eq:ineqtau}  (Noguchi~\cite{Noguchi} only considers the case $\tau=1$, but the proof can be adapted to any $\tau\in\RR$.)

The situation $c_1(L) = 0$, which is not covered by Theorem \ref{th:ExistenceKYM}, provides a simple class of solutions of the K\"ahler--Yang--Mills equations that we consider in Proposition \ref{prop:simple}. We note that even though the extension class of $E$ is non-trivial when $\phi \neq 0$, the existence problem for equivariant solutions of the K\"ahler--Yang--Mills equations is straightforward after dimensional reduction.

The proof of (1) in Theorem \ref{th:ExistenceKYM} is provided by the following (Theorem \ref{th:existencemmap}).

\begin{theorem}\label{th:existencemmap-intro}
Assume that $g(\Sigma) \geqslant 1$  and that $\phi$ is not identically zero. Then, for any constant scalar curvature metric $\omega$ on $\Sigma$ such that
\begin{equation}\label{eq:ineq3-intro}
0 < c_1(L) < \frac{\tau \Vol(\Sigma)}{4\pi}
\end{equation}
holds, there exists $\epsilon > 0$ such that for all $\alpha \in \RR$ with $0 < |\alpha| < \epsilon$ there exists a solution $(\omega_\alpha,h_\alpha)$ of the gravitating vortex equations \eqref{eq:gravvortexeq1}.
\end{theorem}

A proof of Theorem \ref{th:existencemmap-intro} for the case $g(\Sigma) > 1$, is given by using methods similar to those  used by Bradlow~\cite{Brad}. He applied a conformal change $h'=e^{2u}h$ to a suitable hermitian metric $h$ on $L$, observed that the vortex equation for $h'$ is a non-linear PDE for $u\in C^\infty(\Sigma)$ that had already been solved by Kazdan and Warner~\cite{KaWa}, and used their results to obtain an equivalence between the existence of vortices and~\eqref{eq:ineq3-intro}. In the case of the gravitating vortex equations, we can fix a constant scalar curvature metric $g$ on $\Sigma$ and the unique hermitian metric $h$ on $L$ with constant $\Lambda F$, and apply conformal changes to these metrics. The equations for $g'=e^{2u}g, h'=e^{2f}h$, with $u,f\in C^\infty(\Sigma)$, are
\begin{equation}\label{eq:KWtype0}
\begin{split}
\Delta f + \frac{1}{2}e^{2u}(e^{2f}|\phi|^2-\tau) & = - 2\pi c_1(L)/\Vol(\Sigma),\\
\Delta \(u + \alpha e^{2f}|\phi|^2 - 2 \alpha \tau f\) + c(1-e^{2u}) & = 0
\end{split}
\end{equation}
(see Lemma~\ref{lem:KW-equation}). To the knowledge of the authors, these equations have not appeared before in the geometric analysis literature, and provide a promising approach to the gravitating vortex equations.  We prove existence of solutions of \eqref{eq:KWtype0} in the weak coupling limit $0<\lvert\alpha\rvert\ll 1$ by a direct application of the Implicit Function Theorem in Banach spaces. 

Of course the equivalence of equations  \eqref{eq:gravvortexeq} and \eqref{eq:KWtype0} is also valid when $g(\Sigma)=1$, but the  proof of the result in  Theorem \ref{th:existencemmap-intro}  for genus $g(\Sigma)=1$ requires symplectic techniques that we introduce below, extending those used in  \cite{G2} for the usual vortex equations. The K\"ahler--Yang--Mills equations~\eqref{eq:CKYM0} describe the zeros of a moment map for the hamiltonian action of an appropriately extended gauge group on 
the product of the space of connections on the bundle and the space of complex structures on the base manifold (see Section~\ref{subsec:CKYM}). This is very natural considering that these equations are a generalization 
of the conditions of constant scalar curvature for a K\"ahler metric and hermite--Yang--Mills for a connection. Relying on this moment map interpretation, we developed in \cite{AGG} a general theory for the study of coupled equations on K\"ahlerian manifolds which applies to the K\"ahler--Yang--Mills equations. The key consequence of the dimensional reduction mechanism is that the gravitating vortex equations inherit a symplectic interpretation. Using this, we give a proof now valid for $g(\Sigma) \geq 1$ of Theorem \ref{th:existencemmap-intro} showing that the obstruction to deform a vortex over a constant scalar curvature compact Riemann surface corresponds to Hamiltonian Killing vector fields.

The proof of part (2) of Theorem \ref{th:ExistenceKYM} follows from the analysis of the gravitating vortex equations in the case $c = 0$ (note that $\alpha \tau c_1(L) =1$ 
implies this condition, by \eqref{eq:constantcintro}). It follows from~\eqref{eq:KWtype0} that the gravitating vortex equations are very special when $c = 0$, as 
the full system reduces to a single elliptic equation for a function on $\Sigma$ (see~\eqref{eq:single}). 
Furthermore, since $c_1(L) \neq 0$ and $\phi \neq 0$ in the hypothesis of Theorem \ref{th:ExistenceKYM},
\eqref{eq:constantcintro} implies that $\chi(\Sigma)>0$, 
so the only possible topology of $\Sigma$ is the Riemann sphere.  It turns out that, when $c = 0$ and $c_1(L)>0$, the gravitating vortex equations have a physical interpretation, as they are equivalent to the Einstein--Bogomol'nyi equations on a Riemann surface \cite{Yang1992,Yang1994CMP}. Solutions of the Einstein--Bogomol'nyi equations are known in the physics literature as Nielsen--Olesen cosmic strings~\cite{NielsenOlesen}, and describe a special class of solutions of the Abelian Higgs model coupled with gravity in four dimensions \cite{ComtetGibbons,Linet,Linet2}. 
For compact $\Sigma$, the Einstein--Bogomol'nyi equations were studied by Yang~\cite{Yang,Yang3}, who proved a theorem relating the existence of solutions and the relative position of the zeros of the Higgs field. The proof of part (2) in Theorem \ref{th:ExistenceKYM} follows from the equivalence between the Einstein--Bogomol'nyi equations and the gravitating vortex equations for $c = 0$, by a direct application of this
result of Yang. It turns out that Yang's existence theorem (Theorem \ref{th:Yang}) 
can be reformulated in the language of Geometric Invariant Theory (GIT) as developed by Mumford~\cite{MFK}, leading to the statement of (2) in Theorem \ref{th:ExistenceKYM}.


The reformulation of Yang's theorem in GIT terms makes  more transparent the link with algebraic geometry. This relation is illustrated by Theorem~\ref{th:triumph}, where we prove existence result for gravitating vortices on $\PP^1$ --- and hence for the K\"ahler--Yang--Mills equations on $\PP^1 \times \PP^1$ --- for polystable sections nearby a given strictly polystable $\phi$ which admits a solution. For this, we use the moment map interpretation of the K\"ahler--Yang--Mills equations and apply the main result in \cite{GFT}, generalizing a theorem of Sz\'ekelyhidi \cite{Szekelyhidi}. As a particular case of Theorem~\ref{th:triumph}, we recover with different methods a weak version of Yang's existence theorem (Theorem \ref{th:Yang}), around the (symmetric) solution. 
The symplectic techniques used in this paper 
provide also a route to construct obstructions to the existence
of solutions for $g(\Sigma) = 0$, that we consider in a separate paper 
\cite{AGG2}.

A key idea we would like to highlight 
in this paper is 
that the powerful methods of the 
theory of symplectic and GIT quotients are ideally suited to analyze the 
Einstein--Bogomol'nyi equations and, more generally, 
possibly 
other cosmic string solutions in field theory. Cosmic strings are special solutions of physical  field theories describing vortices coupled with gravity, whose energy is concentrated along an infinite line. The basic structure of a string is a complex scalar field --- the Higgs field $\phi$ --- that winds around the location of
the string, where there is a concentration of energy density. In the presence of gauge fields
that interact with the Higgs field and the gravitational field, the string is provided with
a quantized magnetic flux.
The existence of cosmic strings was postulated by Kibble \cite{Kibble} and it is intimately related with the choice of potential for the Higgs field 
$$
V(\phi) = \frac{\lambda}{4}(|\phi|^2 - \tau)^2.
$$ 
The minimum energy configuration has $|\phi|^2 = \tau$ but the phase of $\phi$ is undetermined and labels the points on the manifold of vacua which is a circle. 
The choice of ground state in the circle 
may vary in space-time and, by continuity, $\phi$ can be forced to leave the manifold of vacua 
producing distinguised submanifolds where $\phi = 0$. These submanifolds are known in field theory as topological defects and, among them, cosmic strings are given by a two-dimensional worldsheet in space-time. 

For the Abelian Higgs model coupled with gravity in four dimensions, Linet~\cite{Linet,Linet2} and Comtet and 
Gibbons~\cite{ComtetGibbons}
showed that there exists a critical phase such that when the
cosmological constant vanishes, the existence problem for classical solutions of the 
theory reduces to the
Einstein--Bogomol'nyi equations 
\cite{Yang1992,Yang1994CMP}. 
For non-compact $\Sigma$, the analysis of these equations carried out
during the early 1990s
gave rise to the construction of continuous families of finite-energy 
cosmic string 
solutions~\cite{CHMY,YangSpruck,YangSpruck2,Yang1994,Yang}.
For compact $\Sigma$, the Einstein--Bogomol'nyi equations correspond 
to our gravitating vortex equations in the case $c=0$,
symmetry breaking parameter $\tau>0$ and $\alpha/2\pi$ equal to the
gravitational constant \textemdash{} then the holomorphic global
section $\phi$, the K\"ahler metric $g_\Sigma$, and the Chern
connection of $h$ represent physically the Higgs field, the
gravitational field and the gauge fields, respectively. 



\begin{acknowledgements}
The authors wish to thank O. Biquard, D. J. F. Fox, J. Keller,
N. S. Manton, N. M. Rom\~ao, J. Ross and C. Tipler for useful
discussions. The authors would like to thank also the Isaac Newton
Institute for Mathematical Sciences, Cambridge, for support and
hospitality during the Programme on Moduli Spaces, 2011, when this
work started. Partial support of MGF was provided by QGM (Centre for
Quantum Geometry of Moduli Spaces), Aarhus, funded by the Danish
National Research Foundation, and also by the \'Ecole Polytechnique
F\'ed\'eral de Lausanne. Part of this work was undertaken during a
visit of the second and third authors to the Hausdorff Research
Institute for Mathematics, Bonn, in 2012, and they wish to thank the
hospitality.
\end{acknowledgements}


\section{The K\"ahler--Yang--Mills equations}\label{sec:background}

In this section, we explain briefly some basic facts
from~\cite{AGG,GF} about the K\"ahler--Yang--Mills equations, with emphasis on their symplectic interpretation. 

\subsection{Symplectic origin}\label{subsec:CKYM}

Throughout Section~\ref{subsec:CKYM}, manifolds, bundles, metrics, and
similar objects are of class $C^\infty$. Let $M$ be a compact
symplectic manifold of dimension $2n$, with symplectic form $\omega$
and volume form $\vol_\omega =\frac{\omega^n}{n!}$. Fix a complex
vector bundle $\pi\colon E\to M$ of rank $r$, and a hermitian metric
$H$ on $E$. Consider the positive definite inner product
\[
-\tr\colon\mathfrak{u}(r)\times\mathfrak{u}(r)\lto\RR
\]
on $\mathfrak{u}(r)$. Being invariant under the adjoint U$(r)$-action,
it induces a metric on the (adjoint) bundle $\ad E_H$ of
skew-hermitian endomorphisms of $(E,H)$. Let $\Omega^k$ and
$\Omega^{k}(V)$ denote the spaces of (smooth) $k$-forms and $V$-valued
$k$-forms on $M$, respectively, for any vector bundle $V$ over
$M$. Then the metric on $\ad E_H$ extends to a pairing on the space
$\Omega^\bullet(\ad E_H)$,
\begin{equation}\label{eq:Pairing}
\Omega^p(\ad E_H)\times\Omega^q(\ad E_H)\lto\Omega^{p+q},
\end{equation}
that will be denoted simply $-\tr a_p\wedge a_q$, for
$a_j\in\Omega^j(\ad E_H)$, $j=p,q$. An almost complex structure on $M$
compatible with $\omega$ determines a metric on $M$ and an operator
\begin{equation}
\label{eq:Lambda}
\Lambda\colon\Omega^{p,q}\lto\Omega^{p-1,q-1}
\end{equation}
acting on the space $\Omega^{p,q}$ of smooth $(p,q)$-forms, given by the adjoint of the Lefschetz operator $\Omega^{p-1,q-1}\to\Omega^{p,q}\colon\gamma\mapsto\gamma\wedge \omega$. It can be seen that $\Lambda$ is symplectic, that is, it does not depend on the choice of almost complex structure on $M$. Its
linear extension to adjoint-bundle valued forms will also be denoted
$\Lambda\colon\Omega^{p,q}(\ad E_H)\to\Omega^{p-1,q-1}(\ad E_H)$.

Let $\cJ$ and $\cA$ be the spaces of almost complex structures on $M$
compatible with $\omega$ and unitary connections on $(E,H)$,
respectively; their respective elements will usually be denoted $J$
and $A$. We will explain now how the K\"ahler--Yang--Mills equations
arise naturally in the construction of the symplectic quotient of a
subspace $\cP\subset\cJ\times\cA$ of `integrable pairs'.

The group of symmetries of this theory is the \emph{extended gauge group} $\cX$. Let $E_H$ be the principal U$(r)$-bundle of unitary frames of $(E,H)$. Then, $\cX$ is the group of automorphisms of $E_H$ which cover elements of the group $\cH$ of hamiltonian symplectomorphisms of $(M,\omega)$. There is a canonical short exact sequence of Lie groups
\begin{equation}
\label{eq:coupling-term-moment-map-1}
  1\to \cG \lra{} \cX \lra{\pr} \cH \to 1,
\end{equation}
where $\pr$ maps each $g\in\cX$ into the Hamiltonian symplectomorphism
$\pr(g)\in\cH$ that it covers, and so its kernel $\cG$ is the unitary
gauge group of $(E,H)$, that is, the normal subgroup of $\cX$
consisting of unitary automorphisms covering the identity map on $M$.

There are $\cX$-actions on $\cJ$ and $\cA$, which combine to give an
action on the product $\cJ\times\cA$,
\[
  g(J,A)=(\pr(g)J, gA). 
\]
Here, $\pr(g)J$ denotes the push-forward of $J$ by $\pr(g)$. To
define the $\cX$-action on $\cA$, we view the elements of $\cA$ as
$G$-equivariant splittings $A\colon TE_H\to VE_H$ of the short exact
sequence
\begin{equation}
\label{eq:principal-bundle-ses}
  0 \to VE_H \lto TE_H\lto \pi^*TM \to 0,
\end{equation}
where $VE_H\subset TE_H$ is the vertical bundle on $E_H$. 
Then the $\cX$-action on $\cA$ is
\[
g A \defeq g\circ A \circ g^{-1},
\]
where $g\colon TE\to TE$ denotes the infinitesimal action in the right-hand side.

For each unitary connection $A$, we write $A^\perp y$ for the
corresponding horizontal lift of a vector field $y$ on $M$ to a vector
field on $E_H$. Then each $A\in\cA$ determines a vector-space
splitting of the Lie-algebra short exact sequence
\begin{equation}
\label{eq:Lie-algebras-ses}
0\to\LieG\lto\LieX\lra{\pr}\LieH\to 0
\end{equation}
associated to~\eqref{eq:coupling-term-moment-map-1}, because $A^\perp\eta\in\LieX$ for all $\eta\in\LieH$. Note also that the equation
\begin{equation}\label{eq:hamiltonian-vector-field}
\eta_\varphi\lrcorner\omega=d\varphi
\end{equation}
determines an isomorphism between the space $\LieH$ of Hamiltonian
vector fields on $M$ and the space $C_0^\infty(M,\omega)$ of smooth
functions $\varphi$ such that $\int_M\varphi\vol_\omega=0$, where 
$\vol_\omega\defeq\frac{\omega^n}{n!}$.

The spaces $\cJ$ and $\cA$ have $\cX$-invariant symplectic structures
$\omega_\cJ$ and $\omega_\cA$ induced by $\omega$, that combine to
define a symplectic form on $\cJ\times\cA$, for each non-zero real
constant $\alpha$, given by
\begin{equation}
\label{eq:Sympfamily}
\omega_\alpha = \omega_\cJ + \frac{4 \alpha}{(n-1)!} \omega_\cA.
\end{equation}
The following result provides the starting point for the theory of the
K\"ahler--Yang--Mills equations. This result builds on the moment map
interpretation of the constant scalar curvature equation for a
K\"ahler metric, due to Fujiki \cite{Fujiki} and Donaldson \cite{D1}, and the classical result of Atiyah and Bott \cite{AB}.

\begin{proposition}[{\cite{AGG,GF}}]
\label{prop:momentmap-pairs}
The $\cX$-action on $(\cJ\times \cA,\omega_\alpha)$ is hamiltonian, with $\cX$-equivariant moment map $\mu_{\alpha}\colon \cJ\times \cA\to(\LieX)^*$ given by
\begin{equation}\label{eq:thm-muX}
\begin{split}
\langle \mu_{\alpha}(J,A),\zeta\rangle & = 4i\alpha\int_M \tr A\zeta\wedge(i\Lambda F_A-\lambda \Id)\vol_{\omega}\\
&- \int_M  \varphi\(S_J - \alpha \Lambda^2 \tr F_A \wedge F_A - 4i\lambda \alpha \Lambda \tr F_A\)\vol_{\omega}
\end{split}
\end{equation}
for any $\zeta\in\LieX$ covering $\eta_\varphi \in \LieH$, with $\varphi\in C_0^\infty(M,\omega)$.
\end{proposition}

Here, $F_A$ is the curvature of $A$, $\lambda\in\RR$ is determined by
the topology of the bundle and the cohomology class $[\omega]\in H^2(M,\RR)$, and $S_J$ is the hermitian scalar curvature of $J$. Explicitly,
\[
F_A = - A[A^\perp\cdot,A^\perp\cdot]\in\Omega^2(\ad E_H),
\quad
\lambda=\frac{2\pi nc_1(E)\cdot[\omega]^{n-1}}{r[\omega]^n},
\]
with the convention $2\pi c_1(E)=[i\tr F_A]$. A key observation in
\cite{AGG,GF} is that the space $\cJ \times \cA$ has a (formally
integrable) complex structure $\mathbf{I}$ preserved by the
$\cX$-action, given by
\begin{equation}
\label{eq:complexstructureI}
\mathbf{I}_{\mid(J,A)}(\gamma,a) = (J\gamma,-a(J \cdot)),
\text{ for } (\gamma,a) \in T_J\cJ \oplus T_A\cA.
\end{equation}
For positive $\alpha$, $\mathbf{I}$ is compatible with the family of
symplectic structures \eqref{eq:Sympfamily}, and so it defines
K\"ahler structures on $\cJ\times\cA$. The condition $\alpha>0$ will
be assumed in the sequel.

Suppose now that there exist K\"{a}hler structures on $M$ with
K\"{a}hler from $\omega$. This means the subspace $\cJi \subset \cJ$
of integrable almost complex structures compatible with $\omega$ is
not empty. For each $J\in \cJi$, let $\cA^{1,1}_J\subset\cA$ be the
subspace of connections $A$ with $F_A \in \Omega_J^{1,1}(\ad E_H)$,
where $\Omega^{p,q}_J$ is the space of $(p,q)$-forms with respect to
$J$. Then the space of \emph{integrable pairs}
\begin{equation}
\label{eq:cP}
  \cP\subset \cJ\times \cA,
\end{equation}
consisting of elements $(J,A)$ with $J\in\cJi$ and $A\in\cA^{1,1}_J$,
is a $\cX$-invariant (possibly singular) K\"ahler submanifold. The
zero locus of the induced moment map $\mu_\alpha$ for the $\cX$-action
on $\cP$ corresponds precisely to the solutions of the (coupled)
\emph{K\"ahler--Yang--Mills equations}
\begin{equation}\label{eq:CKYM1}
\begin{split}
i \Lambda F_A & = \lambda \Id,\\
S_J \; - \; \alpha \Lambda^2 \tr F_A \wedge F_A & = c.
\end{split}
\end{equation}
Here, $S_J$ is the scalar curvature of the metric
$g_J=\omega(\cdot,J\cdot)$ and the constant $c\in\RR$ depends on
$\alpha$, the cohomology class of $\omega$ and the topology of $M$ and
$E$ (see \cite[Section 2]{AGG}).

One can express the K\"ahler--Yang--Mills equations from an alternative point of view in which we fix a compact complex manifold $X$ of dimension $n$, a K\"ahler class $\Omega \in
H^{1,1}(X)$ and a holomorphic vector bundle $E$ over $X$. Then these
equations, for a fixed constant parameter $\alpha\in\RR$, are
\begin{equation}\label{eq:CKYM2}
\begin{split}
i\Lambda_\omega F_H & =\lambda\Id,\\
S_\omega\;-\;\alpha\Lambda_\omega^2\tr F_H\wedge F_H & =c,
\end{split}
\end{equation}
where the unknowns are a K\"ahler metric on $X$ with K\"ahler form
$\omega$ in $\Omega$, and a hermitian metric $H$ on $E$. In this case,
$F_H$ is the curvature of the Chern connection $A_H$ of $H$ on $E$,
and $S_\omega$ is the scalar curvature of the K\"ahler metric. Note
that the operator in~\eqref{eq:Lambda} depends on $\omega$, and the
constant $c\in\RR$ depends on $\alpha$, $\Omega$ and the topology of
$X$ and $E$.

\section{The gravitating vortex equations}\label{sec:gravvort}

In this section we introduce the gravitating vortex equations, which are, in a sense, the
main object of study of the present paper. We shall derive the equations by
equivariant dimensional reduction of the K\"ahler--Yang--Mills equations on
$\Sigma \times \PP^1$.

\subsection{Gravitating vortices}\label{subsec:gravvort}

Let $\Sigma$ be a compact connected Riemann surface of arbitrary genus. Let $L$ be a holomorphic line bundle over $L$ and $\phi \in H^0(\Sigma,L)$ a holomorphic section of $L$. We fix a \emph{symmetry breaking parameter} $0 < \tau \in \RR$ and a \emph{coupling constant} $\alpha \in \RR$.

\begin{definition}
The \emph{gravitating vortex equations}, for a K\"ahler metric on $\Sigma$ with K\"ahler form $\omega$ and a hermitian metric $h$ on $L$, are

\begin{equation}\label{eq:gravvortexeq1}
\begin{split}
i\Lambda_\omega F_h + \frac{1}{2}(|\phi|_h^2-\tau) & = 0,\\
S_\omega + \alpha(\Delta_\omega + \tau) (|\phi|_h^2 -\tau) & = c.
\end{split}
\end{equation}
\end{definition}

Here, $S_\omega$ is the scalar curvature of $\omega$, $F_h$ stands for the curvature of the Chern connection of $h$, $|\phi|_h^2$ is the smooth function on $\Sigma$ given by the norm-square of $\phi$ with respect to $h$ and $\Delta_\omega$ is the Laplace operator for the metric $\omega$, defined by
$$
\Delta_\omega f = 2i \Lambda_\omega \dbar \partial f, \qquad \textrm{ for } f \in C^\infty(\Sigma).
$$

The constant $c \in \RR$ is topological, and is explicitly given by
\begin{equation}\label{eq:constantc}
c = \frac{2\pi(\chi(\Sigma) - 2\alpha\tau c_1(L))}{\Vol_\omega(\Sigma)},
\end{equation}
as can be deduced by integrating the equations.
The gravitating vortex equations for $\phi = 0$, are the condition that
$\omega$ is a constant scalar curvature K\"ahler metric on $\Sigma$ and $h$ is
a hermite-Einstein metric on $L$. By the Uniformization Theorem for Riemann
surfaces, the existence of these `trivial 
solutions' reduces by Hodge Theory to the condition $c_1(L) = \tau\Vol_\omega(\Sigma)/4\pi$.

Excluding this trivial case, the sign of $c$ plays an important role in the
existence problem for the gravitating vortex equations. For instance, for
$\alpha > 0$ and $c$ positive, 
the existence of a solution of
\eqref{eq:gravvortexeq1} with $\phi$ not identically zero forces the topology
of the surface to be that of 
the $2$-sphere, 
as $c_1(L)\geq 0$ 
implies the
positivity of the Euler characteristic $\chi(\Sigma) > 0$. Further, when $c = 0$ and $c_1(L) > 0$, the only possible topology is also the $2$-sphere. 
This important case, related to the Einstein--Bogomol'nyi equations and the physics of cosmic strings in the 
Abelian Higgs model, will be treated separately in Section \ref{sec:cosmicstring}.

\begin{remark}\label{rem:c1=0}
When $c = 0$ and $c_1(L) = 0$ the existence of gravitating vortex with $\phi \neq 0$ implies that $L$ is isomorphic to the trivial line bundle $\mathcal{O}_\Sigma$, $\phi$ is constant with $|\phi|_h^2 = \tau$, and the gravitating vortex equations are the condition that $\omega$ is a flat K\"ahler metric on an elliptic curve $\Sigma$ and $h$ is flat. This case will be discussed further in relation to the K\"ahler--Yang--Mills equations in Proposition \ref{prop:simple}.
\end{remark}


For a fixed K\"ahler metric $\omega$, the first equation in
\eqref{eq:gravvortexeq1} corresponds to the (Bogomol'nyi Abelian) vortex equation 
\begin{equation}\label{eq:vortexeq}
i\Lambda_\omega F_h + \frac{1}{2}(|\phi|_h^2-\tau) = 0,
\end{equation}
for a hermitian metric $h$ on $L$. In \cite{Noguchi,Brad,G1,G2} Noguchi, Bradlow 
and the third author gave independently and with different methods a complete characterization of the existence of \emph{Abelian vortices} on a compact Riemann surface, that is, of solutions of the equations \eqref{eq:vortexeq}.

\begin{theorem}[\cite{Brad,G1,G2}]\label{th:B-GP}
Assume that $\phi$ is not identically zero. For every fixed K\"ahler form $\omega$, there exists a unique solution $h$ of the vortex equations \eqref{eq:vortexeq} if and only if
\begin{equation}\label{eq:ineq}
c_1(L) < \frac{\tau \Vol(\Sigma)}{4\pi}.
\end{equation}
\end{theorem}

Inspired by  work of Witten \cite{Witten} and Taubes \cite{Taubes}, the method in \cite{G1} exploited  the dimensional reduction of the hermitian--Yang--Mills equations from four to two dimensions, combined with the theorem of Donaldson, Uhlenbeck and Yau \cite{D3,UY}. We will provide with some details of this method in Section \ref{sec:dimenred}, in order to derive equations \eqref{eq:gravvortexeq1} from the K\"ahler--Yang--Mills equations.


\subsection{The gravitating vortex equations from dimensional reduction}\label{sec:dimenred}

In this section we derive the gravitating vortex equations
\eqref{eq:gravvortexeq1} as dimensional reduction of the K\"ahler--Yang--Mills
equations \eqref{eq:CKYM2} on a rank two bundle $E$ over the product of a
compact, connected, Riemann surface $\Sigma$ with the Riemann sphere $\PP^1$.

As in the previous section, we fix a compact connected Riemann surface $\Sigma$ of arbitrary genus, a holomorphic line bundle $L$ over $\Sigma$ and $\phi \in H^0(\Sigma,L)$ a holomorphic section of $L$. There is canonically associated to $(L,\phi)$ a rank two holomorphic vector bundle $E$ over
$$
X = \Sigma \times \PP^1
$$ 
given as an extension
\begin{equation}\label{eq:bundleE}
0 \to p^*L \lto E \lto q^*\mathcal{O}_{\PP^1}(2) \to 0.
\end{equation}
Here $p$ and $q$ are the projections from $X$ to $\Sigma$ and $\PP^1$ respectively. By $\mathcal{O}_{\PP^1}(2)$ we denote as usual the holomorphic line bundle with Chern class $2$ on $\PP^1$, isomorphic to the holomorphic tangent bundle of $\PP^1$. Extensions as above are parametrized by
$$
H^1(X,p^*L \otimes q^*\cO_{\PP^1}(-2)) \cong H^0(\Sigma,L) \otimes H^1(\PP^1,\mathcal{O}_{\PP^1}(-2)) \cong H^0(\Sigma,L),
$$
and we choose $E$ to be the extension determined by $\phi$.

Let $\SU(2)$ act on $X$, trivially on $\Sigma$, and in the standard way on
$\PP^1 \cong \SU(2)/\U(1)$. This action can be lifted to trivial actions on 
$E$ and  $p^*L$ and the standard action on $\mathcal{O}_{\PP^1}(2)$. Since the induced actions on $H^0(\Sigma,L)$ and $H^1(\PP^1,\mathcal{O}_{\PP^1}(-2)) \cong H^0(\PP^1,\mathcal{O}_{\PP^1})^* \cong \CC$ are trivial, $E$ is an $\SU(2)$-equivariant holomorphic vector bundle over $X$.

For $\tau \in \RR_{>0}$, consider the $\SU(2)$-invariant K\"ahler metric on $X$ whose K\"ahler form is
\begin{equation}\label{eq:omegatau}
\omega_\tau = p^*\omega + \frac{4}{\tau} q^*\omega_{FS},
\end{equation}
where $\omega$ is a K\"ahler form on $\Sigma$ and $\omega_{FS}$ is the Fubini--Study metric on $\PP^1$, given in homogeneous coordinates by
$$
\omega_{FS} = \frac{i dz \wedge d\overline z}{(1+|z|^2 )^2}
$$
and such that $\int_{\PP^1}\omega_{FS} = 2\pi$.

We can now state the main result of this section. Throughout, we will assume
that the coupling constants $\alpha$ in 
\eqref{eq:gravvortexeq1} 
and \eqref{eq:CKYM2} coincide.

\begin{proposition}\label{prop:dimred}
The triple $(\Sigma,L,\phi)$ admits a solution $(\omega,h)$ of the gravitating vortex equation \eqref{eq:gravvortexeq} with parameter $\tau$ if and only if $(X,E)$ admits an $\SU(2)$-invariant solution of the K\"ahler--Yang--Mills equations \eqref{eq:CKYM2} with K\"ahler form $\omega_\tau= p^*\omega + \frac{4}{\tau} q^*\omega_{FS}$.
\end{proposition}

\begin{proof}
We start recalling a few facts from the proof of \cite[Proposition 12]{G1}. Firstly, the class of the extension $E$ is represented by the $\SU(2)$-invariant element
$$
\beta = p^*\phi \otimes q^*\eta \in \Omega^{0,1}(X,\Hom(q^*\cO_{\PP^1}(2),p^*L)),
$$
where, in homogeneous coordinates,
$$
\eta = \frac{\sqrt{8\pi}}{\tau}\frac{dz \otimes d\overline{z}}{(1 + |z|^2)^2} \in \Omega^{0,1}(\PP^1,\Hom(\cO_{\PP^1}(2),\cO_{\PP^1})).
$$
Let $H$ be an $\SU(2)$-invariant hermitian metric on $E$. Since $H$ is $\SU(2)$-invariant and the actions of $\SU(2)$ on $p^*L$ and $q^*\mathcal{O}_{\PP^1}(2)$ correspond to different weights, it has to be of the form 
$$
H = \mathbf{h}_1 \oplus \mathbf{h}_2 
$$
with respect to the natural smooth splitting $E \cong p^*L \oplus q^*\cO_{\PP^1}(2)$, for hermitian metrics $\mathbf{h}_1$ on $p^*L$ and $\mathbf{h}_2$ on $q^*\cO_{\PP^1}(2)$. Moreover, we can assume
$$
\mathbf{h}_1 = h_1, \qquad  \mathbf{h}_2 = h_2 \otimes \frac{8\pi}{\tau}\frac{dz \otimes d\overline z}{(1+|z|^2)^2}
$$ 
for a hermitian metric $h_1$ on $L$ and $h_2$ on $\cO_X$. Using this, the Chern connection of the invariant metric $H$ is of the form
\[ 
A_H = \left( \begin{array}{cc}
A_{h_1} & \beta \\
- \beta^* & A_{\mathbf{h}_2}
\end{array} \right)
\]
and the correspoding curvature is
 \[ 
F_H = \left( \begin{array}{cc}
F_{h_1} - \beta \wedge \beta^* & D'\beta \\
- D'' \beta^* & F_{\mathbf{h}_2} - \beta^* \wedge \beta
\end{array} \right),
\]
where the operators $D'$ and $D''$ are given by the $(1,0)$-valued and $(0,1)$-valued parts of the connection $D = D' + D''$ on $p^*L \otimes q^*\cO_{\PP^1}(-2)$ induced by $A_{h_1}$ and $A_{\mathbf{h}_2}$.

We go now for the proof of the statement: it will reduce to \cite[Proposition 12]{G1} once we calculate the term 
\begin{align*}
\Lambda^2_{\omega_\tau}\tr F_H \wedge F_H & = \Lambda^2_{\omega_\tau}\((F_{h_1} - \beta \wedge \beta^*)^2 - 2 \partial \beta \wedge \dbar \beta^* + (F_{\mathbf{h}_2} - \beta^* \wedge \beta)^2\)
\end{align*}
appearing in the K\"ahler--Yang--Mills equations \eqref{eq:CKYM2}.

\begin{lemma}
Consider the hermitian metric on $L$ given by $h = h_1 \otimes h_2^{-1}$. Then, 
\begin{equation}\label{eq:lemma}
\Lambda^2_{\omega_\tau}\tr F_H \wedge F_H = - \Delta_\omega |\phi|^2_h - 2i\tau \Lambda_\omega F_{h_2}.
\end{equation}
\end{lemma}

To prove \eqref{eq:lemma}, recall from the proof of \cite[Proposition 12]{G1} the formulae

\begin{align*}
\beta \wedge \beta^* &= \frac{i}{\tau}|\phi|^2_h \omega_{FS}\\
D' \beta \wedge D'' \beta^* & = \partial \phi \otimes \eta \wedge \dbar \phi^* \otimes \eta^*\\
& = \partial \phi\wedge \dbar \phi^* \otimes  \eta \wedge \eta^* = \frac{i}{\tau}\partial \phi\wedge \dbar \phi^* \otimes  \omega_{FS}.
\end{align*}
Using now
$$
F_{\mathbf{h}_2} = F_{h_2} - 2i\omega_{FS}
$$
combined with the Weitzenb\"ock-type formula
$$
\Delta |\phi|^2_h = 2 i \Lambda_\omega \dbar \partial |\phi|^2_h = 2i\Lambda_\omega(\partial \phi \wedge \dbar \phi^* + |\phi|^2_hF_h),
$$
we obtain the desired identity \eqref{eq:lemma}
\begin{align*}
\Lambda^2_{\omega_\tau}\tr F_H \wedge F_H & = \Lambda^2_{\omega_\tau}\(-\frac{i}{2}|\phi|^2_h(F_{h_1} - F_{h_2})  -\frac{i\tau}{2}F_{h_2} - \frac{i}{2}\partial \phi \wedge \dbar \phi^*\) \otimes \frac{4}{\tau}\omega_{FS}\\
& = - \Delta |\phi|^2_h - 2i\tau \Lambda F_{h_2}.
\end{align*}
With formula \eqref{eq:lemma} at hand, we can prove now the statement of the Proposition. Suppose first that $(\omega_\tau,\mathbf{h})$ is a solution of the K\"ahler--Yang--Mils equations \eqref{eq:CKYM0}. Then we have
\begin{align*}
i\Lambda_\omega F_{h_1} + \frac{1}{4}|\phi|^2_h & = \lambda,\\
i\Lambda_\omega F_{h_2} - \frac{1}{4}|\phi|^2_h + \frac{\tau}{2} & = \lambda,\\
S_\omega + \alpha (\Delta_\omega |\phi|^2_h + 2i\tau \Lambda_\omega F_{h_2}) & = c.
\end{align*}
Substracting the first two equations and using the second equality on the third equation we obtain the result.

On the other hand, if $(\omega,h)$ is a solution of the gravitating vortex equations \eqref{eq:gravvortexeq1} we have to solve the system
\begin{align*}
i\Lambda_\omega F_{h_1} + \frac{1}{4}|\phi|^2_h & = \lambda,\\
i\Lambda_\omega F_{h_2} - \frac{1}{4}|\phi|^2_h + \frac{\tau}{2} & = \lambda,\\
\Lambda_{\omega_{FS}}D'\beta & = 0\\
\Lambda_{\omega_{FS}} D''\beta^* & = 0\\
S_\omega + \alpha (\Delta_\omega |\phi|^2_h + 2i\tau \Lambda_\omega F_{h_2}) & = c.
\end{align*}
As in \cite[Proposition 12]{G1}, the third and fourth equations are automatically satisfied while the rest are solved by the ansatz $h_2 = e^f$, $h_1 = h \otimes h_2$ and 
$$
\qquad \qquad \qquad \qquad \qquad \qquad \qquad \Delta_\omega f = 2i \Lambda_\omega F_{h_2} = 2\lambda + \frac{1}{2}|\phi|^2_h - \tau.  \qquad \qquad \qquad \qquad \qedhere
$$
\end{proof}


When $c_1(L) = 0$, the gravitating vortex equations are solved trivially by a constant scalar curvature K\"ahler metric on $\Sigma$ and a flat hermitian metric on $L$ (cf. Remark \ref{rem:c1=0}). In addition, if $\phi \neq 0$ the holomorphic bundle $E$ is of the form
$$
0 \to p^*\mathcal{O}_{\Sigma} \lto E \lto q^*\mathcal{O}_{\PP^1}(2) \to 0,
$$
with non-trivial extension class, and the previous fact combined with Proposition \ref{prop:dimred} provides a straighforward existence result for the K\"ahler--Yang--Mills equations, that we state in the next proposition.

\begin{proposition}\label{prop:simple}
If $c_1(L) = 0$ and $\phi \neq 0$ there exists a unique $\SU(2)$-invariant solution of the K\"ahler--Yang--Mills equations with K\"ahler class $[\omega_\tau]$, for any given $\tau > 0$ and $\alpha \in \RR$.
\end{proposition}

\begin{proof}
By Proposition \ref{prop:dimred}, a solution is provided by taking $\omega$ the unique constant scalar curvature metric on $\Sigma$ with volume fixed by $[\omega_\tau]$ and $h$ the unique flat hermitian metric such that $\tau = |\phi|^2_h$. Given a different solution of the gravitating vortex equations $(\omega',h')$ inducing the same K\"ahler class $[\omega_\tau]$, by uniqueness of solutions of the vortex equations (Theorem \ref{th:B-GP}) it follows that $h'$ is flat and hence $\tau = |\phi|^2_{h'}$. We conclude that $\omega'$ has constant scalar curvature and therefore $\omega' = \omega$, which implies $h = h'$.
\end{proof}

\section{Existence of gravitating vortices in the weak coupling limit}

In this section we fix the symmetry breaking parameter $\tau > 0$, but allow
the coupling constant $\alpha$ to vary. The goal is to prove an existence
theorem for the gravitating vortex equations \eqref{eq:gravvortexeq1} in the
\emph{weak coupling limit} $0 < |\alpha| \ll 1$, by deforming solutions with
coupling constant $\alpha = 0$. 
For $\alpha = 0$, the equations \eqref{eq:gravvortexeq1} are equivalent to 
the condition that $\omega$ is a constant scalar curvature K\"ahler metric on $\Sigma$ (which exists by the Uniformization Theorem) and $h$ is a solution of the vortex equations \eqref{eq:vortexeq}.

\begin{theorem}\label{th:existencemmap}

Assume that $g(\Sigma) \geqslant 1$  and that $\phi$ is not identically zero. Then, 
for any constant scalar curvature metric $\omega$ on $\Sigma$ such that
\begin{equation}\label{eq:ineq3}
0 < c_1(L) < \frac{\tau \Vol_\omega(\Sigma)}{4\pi}
\end{equation}
holds, there exists $\epsilon > 0$ such that for all $\alpha \in \RR$ with $0 < |\alpha| < \epsilon$ there exists a solution $(\omega_\alpha,h_\alpha)$ of the gravitating vortex equations \eqref{eq:gravvortexeq1}.
\end{theorem}

Our main tool to prove this result will be the Implicit Function Theorem in 
Banach spaces. As a  consequence of this theorem, as $|\alpha| \to 0$, the 
pair $(\omega_\alpha,h_\alpha)$ converges uniformly to $(\omega,h')$, 
where $h'$ is the unique solution of the Abelian vortex equations 
\eqref{eq:vortexeq} with respect to $\omega$, as stated in Theorem \ref{th:B-GP}.

\subsection{Kazdan--Warner type equations and the weak coupling limit}
\label{sec:KazdanWarner}

In this section we deduce an equivalent formulation of the gravitating 
vortex equations in terms of Kazdan--Warner type equations.
This is inspired by Bradlow's approach \cite{Brad}, where the vortex equations are reduced  
to a differential equation for a single function 
on the Riemann surface, already studied by Kazdan and Warner \cite{KaWa}.
We use this, and apply the Implicit Function Theorem in Banach spaces,
to prove the result in Theorem \ref{th:existencemmap} when $g(\Sigma)>1$.

To obtain the Kazdan--Warner type equations, let $L$ be a holomorphic line 
bundle over $\Sigma$ and $\phi\in H^0(L)$. 
Let $\omega$ be a constant scalar curvature metric on $\Sigma$ and $h$ the 
unique hermite--Einstein metric on $L$ with respect to $\omega$, that is,
$$
S_\omega = 2\pi\chi(\Sigma)/\Vol_\omega(\Sigma), \qquad i \Lambda_\omega F_h = 2\pi c_1(L)/\Vol_\omega(\Sigma).
$$ 
Given smooth functions $u,f$ on $\Sigma$, consider the conformal rescaling $\omega' = e^{2u} \omega$, $h' = e^{2f} h$. In the sequel, we denote by $\Delta$ the Laplacian of the constant scalar curvature metric $\omega$ and also set $|\phi|^2 = |\phi|_h^2$.

\begin{lemma}\label{lem:KW-equation}
The pair $(\omega',h')$ satisfies the gravitating vortex equations \eqref{eq:gravvortexeq1} if and only if
\begin{equation}\label{eq:KWtype}
\begin{split}
\Delta f + \frac{1}{2}e^{2u}(e^{2f}|\phi|^2-\tau) & = - 2\pi c_1(L)/\Vol(X),\\
\Delta \(u + \alpha e^{2f}|\phi|^2 - 2 \alpha \tau f\) + c(1-e^{2u}) & = 0,
\end{split}
\end{equation} 
where $c$ is given by \eqref{eq:constantc}.
\end{lemma}

\begin{proof}
Using the identities
$$
e^{2u}S_{\omega'} = S_\omega + \Delta u, \qquad e^{2u}i\Lambda_{\omega'}F_{h'} = i\Lambda_\omega F_h + \Delta f,
$$
we obtain the system
\begin{equation}
\begin{split}
i\Lambda_\omega F_h + \Delta f + \frac{1}{2}e^{2u}(e^{2f}|\phi|^2-\tau) & = 0,\\
S_g + \Delta u + \alpha\Delta(e^{2f}|\phi|^2 -\tau) - 2 \alpha \tau (i\Lambda_\omega F_h + \Delta f) - ce^{2u} & = 0,
\end{split}
\end{equation}
and the result now follows from the choice of $\omega$ and $h$.
\end{proof}

To start our deformation argument, we note that for $\alpha = 0$ the constant
function $u_0 = 0$ provides a solution of the second equation in
\eqref{eq:KWtype}. With this ansatz, following Bradlow \cite{Brad}, we note
that by the Kazdan--Warner Theorem \cite{KaWa} 
there exists a unique smooth solution  $f_0 \in C^\infty(\Sigma)$ of the first equation 
$$
\Delta f_0 + \frac{1}{2}(e^{2f_0}|\phi|^2-\tau) = - 2\pi c_1(L)/\Vol(X)
$$ 
if and only the inequality \eqref{eq:ineq} holds. In this case, the metric
$h' = e^{2f_0}h$ provides the unique solution of the Abelian vortex equations
\eqref{eq:vortexeq}. Hence, assuming \eqref{eq:ineq}, 
the linearization of the equations \eqref{eq:KWtype} 
with $\alpha = 0$ at $(0,f_0)$ is given by
\begin{equation}\label{eq:KWtypelin}
\begin{split}
\Delta \dot f + \dot u (e^{2f_0}|\phi|^2-\tau) + \dot f e^{2f_0}|\phi|^2 & = 0,\\
\Delta \dot u 
 - \frac{4\pi\chi(\Sigma)}{\Vol_\omega(\Sigma)}\dot u & = 0,
\end{split}
\end{equation} 
for $\dot u, \dot f \in C^\infty(\Sigma)$.

\begin{lemma}\label{lem:injectiveKW}
If the genus of $\Sigma$ is strictly bigger than one, the only solution of \eqref{eq:KWtypelin} is the trivial solution.
\end{lemma}

\begin{proof}
By assumption, we have $\chi(\Sigma) < 0$ and hence
$$
\Delta \dot u = \frac{4\pi\chi(\Sigma)}{\Vol_\omega(\Sigma)}\dot u
$$
only admits the trivial solution, since $\Delta$ is a positive operator. Hence, the first equation in \eqref{eq:KWtypelin} reduces to
$$
\Delta \dot f + \dot f e^{2f_0}|\phi|^2 = 0,
$$
and again by positivity of $\Delta$ we obtain
$$
0 \geq - \int_\Sigma |\dot f|^2 e^{2f_0}|\phi|^2 = \int_\Sigma \dot f  \Delta \dot f \omega = \int_\Sigma |d\dot f|^2\omega \geq 0,
$$
and consequently $\dot f$ has to be constant. Since $\phi$ is non-identically zero, we necessarily have $\dot f = 0$.
\end{proof}

To prove Theorem \ref{th:existencemmap} when $g(\Sigma)>1$,
consider the operator
$$
\mathbf{L} \colon C^\infty(\Sigma) \oplus C^\infty(\Sigma) \oplus \RR \to C^\infty(\Sigma) \oplus C^\infty(\Sigma)
$$
defined by $\mathbf{L} = \mathbf{L}_1 \oplus \mathbf{L}_2$ with 
\begin{align*}
\mathbf{L}_1(u,f,\alpha) & = \Delta f + \frac{1}{2}e^{2u}(e^{2f}|\phi|^2-\tau) + 2\pi c_1(L)/\Vol(X),\\
\mathbf{L}_2(u,f,\alpha) & = \Delta \(u + \alpha e^{2f}|\phi|^2 - 2 \alpha \tau f\) + c(1-e^{2u}),
\end{align*}
where $c = c(\alpha)$ is given by \eqref{eq:constantc}. By Lemma \ref{lem:injectiveKW}, the linearization $\delta \mathbf{L}_0$ of $\mathbf{L}$ with respect to $u$ and $f$ at $(0,f_0,0)$, given by \eqref{eq:KWtypelin}, is injective. We check now that it is an isomorphism. For this, we note that $\Delta - \frac{4\pi\chi(\Sigma)}{\Vol_\omega(\Sigma)}$ is a self-adjoint and positive operator, and hence it is invertible. To prove surjectivity of $\delta \mathbf{L}_0$, it is therefore enough to show that the self-adjoint operator  $\Delta +  e^{2f_0}|\phi|^2$ is invertible, which again follows by positivity.  The result follows now, by the Implicit Function Theorem on Banach spaces, by taking suitable Sobolev completions of $C^\infty(\Sigma)$. Regularity of the solutions follows easily from a boot-strapping argument applied to \eqref{eq:KWtype} and regularity of the Laplacian.

\begin{remark}
In the case $g(\Sigma) \leq 1$, there exist non-trivial solutions of 
\eqref{eq:KWtypelin} that obstruct the deformation argument. These 
solutions are given by inverting the self-adjoint positive operator 
$\Delta + e^{2f_0}|\phi|^2$ at the function $-\dot u (e^{2f_0}|\phi|^2-\tau)$, 
for any non-zero eigenfunction $\dot u$ of the Laplacian $\Delta$ with eigenvalue 
\begin{equation}\label{eq:firsteigenvalue}
\frac{4\pi\chi(\Sigma)}{\Vol_\omega(\Sigma)}.
\end{equation}
For $g(\Sigma) = 1$, the space of solutions is parametrized by $\RR$. 
For $g(\Sigma) = 0$, by a theorem of Lichnerowicz \cite{Lich} (see also 
\cite[Proposition 2.6.2]{Gauduchon}) the constant \eqref{eq:firsteigenvalue} 
is precisely the first non-zero eigenvalue of the Laplacian, and the 
corresponding eigenfunctions are symplectic potentials for Hamiltonian 
Killing vector fields on the round sphere of volume $\Vol_\omega(\Sigma)$. 
A geometric interpretation of this fact will be provided by 
the moment map framework in the next section.
\end{remark}

\subsection{Existence in the weak coupling limit revisited}\label{sec:weakbis}

From the dimensional reduction argument and the moment map interpretation of
the K\"ahler--Yang--Mills equations \cite{AGG,GF} (see Section
\ref{subsec:CKYM}), it follows that the gravitating vortex equations
\eqref{eq:gravvortexeq} also have a moment map interpretation. We use now this
fact to give a proof (covering also the case $g(\Sigma)=1$) of 
Theorem \ref{th:existencemmap}.

Let $S$ be compact, connected, oriented, smooth surface endowed with a symplectic form $\omega$. Let $(L,h)$ be a hermitian line bundle over $S$. As in Section \ref{subsec:CKYM}, consider the \emph{extended gauge group} $\cX$ of $(L,h)$ and $(S,\omega)$, given by an extension
\begin{equation}
\label{eq:coupling-term-moment-map-S}
  1\to \cG \lra{} \cX \lra{\pr} \cH \to 1,
\end{equation}
of the group $\cH$ of Hamiltonian symplectomorphisms of $(S,\omega)$ by the unitary gauge group $\cG$ of $(L,h)$. Consider the space of integrable triples 
$$
\mathcal{T} \subset \cJ \times \cA \times \Omega^0(L)
$$
defined by the equation
$$
\mathcal{T} = \{(J,A,\phi): \quad \dbar_A \phi = 0\}.
$$
Let $(J,A,\phi) \in \cT$ and denote by $\Sigma$ the Riemann surface determined by $J$.  We adapt the method in \cite[Sect. 4]{AGG}, considering deformations of $(\omega,h)$ which preserve the volume $\Vol(\Sigma)$ of $\omega$ (i.e. the K\"ahler class) in the fixed Riemann surface $\Sigma$. Consider the space $C^\infty_0(\Sigma)$ of smooth functions $\varphi$ on $\Sigma$ such that $\int_\Sigma \varphi \omega = 0$, endowed with the $L^2$-pairing
$$
\langle \varphi_0,\varphi_1 \rangle = \int_\Sigma \varphi_0 \varphi_1 \omega.
$$ 
Let $U \subset C_0^\infty(\Sigma)$ be an open neighbourhood of $0$ (in $C^k$-norm, say) such that $\widetilde \omega = \omega + 2i \partial \dbar \varphi$ is positive for all $\varphi \in U$. Consider the operator
$$
\mathbf{B} \colon U \oplus C^\infty(\Sigma) \oplus \RR \to C^\infty(\Sigma) \oplus C_0^\infty(\Sigma)
$$
defined by $\mathbf{B} = \mathbf{B}_1 \oplus \mathbf{B}_2$ with 
\begin{align*}
\mathbf{B}_1(\varphi,f,\alpha) & = i\Lambda_{\widetilde \omega} F_{\widetilde h} + \frac{1}{2}|\phi|^2_{ \widetilde h} - \frac{\tau}{2},\\
\mathbf{B}_2(\varphi,f,\alpha) & = S_{\widetilde \omega} + \alpha \Delta_{\widetilde \omega} |\phi|^2_{ \widetilde h} - 2\alpha\tau i\Lambda_{\widetilde \omega} F_{\widetilde h} - c(\alpha),
\end{align*}
where $\widetilde h = e^{2f}h$ and $c = c(\alpha)$ is given by \eqref{eq:constantc}.  Let $\delta \mathbf{B}^0$ denote the linearization of $\mathbf{B}$ with respect to $(\varphi,f)$ at $(0,0,0)$. To give a formula for $\delta \mathbf{B}^0 = \delta \mathbf{B}^0_1 \oplus \delta \mathbf{B}^0_2$,  we introduce the notation
$$
P \colon C^\infty_0(\Sigma) \to T_J\cJ: \varphi \mapsto - L_{\eta_\varphi}J
$$
for the infinitesimal action of $\eta_\varphi \in \Lie \cH$, and $P^*$ for the adjoint with respect to the metric on $T_J\cJ$ and the $L^2$-pairing on $C_0^\infty(\Sigma)$.

\begin{lemma}
\begin{equation}\label{eq:deltaB}
\begin{split}
\delta \mathbf{B}^0_1(\varphi,f) & = \Delta_\omega f + d^*(\eta_\varphi \lrcorner iF_h) +  f|\phi|_h^2 + J\eta_\varphi \lrcorner d(i\Lambda_{\omega} F_{h} + \frac{1}{2}|\phi|^2_{h})),\\
\delta \mathbf{B}^0_2(\varphi,f) & = - P^*P \varphi - (dS_\omega,d\varphi)_\omega.
\end{split}
\end{equation}
\end{lemma}

The proof is analogue to the proof of \cite[Proposition 4.7]{AGG}. Note that $P^*P$ is, up to a multiplicative constant factor, the Lichnerowicz operator of the compact K\"ahler manifold $(\Sigma,\omega)$ (see e.g. \cite{LS1})
$$
- P^*P \varphi = \Delta_\omega^2 \varphi - S_\omega \Delta_\omega \varphi + (dS_\omega,d\varphi)_\omega.
$$
This is an elliptic self-adjoint semipositive differential operator of order $4$, whose kernel is the set of functions $\varphi$ such that $\eta_\varphi$ is a Killing Hamiltonian vector field, and which may be interpreted as the linearization of the constant scalar curvature K\"ahler equation at $\omega$.

\begin{lemma}\label{lem:deltaBcscK}
Assume that $\omega$ has constant scalar curvature and $\phi \neq 0$. Then, the kernel of $\delta \mathbf{B}^0$ can be identified with the space of Hamiltonian Killing vector fields on $(\Sigma,\omega)$.
\end{lemma}

\begin{proof}
By \eqref{eq:deltaB}, $\delta \mathbf{B}^0(\varphi,f) = 0$ implies that
$\varphi$ is the potential of a Hamiltonian Killing vector field. The claim
follows by positivity of the operator $\Delta_\omega + |\phi|_h^2$, as argued
in Section \ref{sec:KazdanWarner}.
\end{proof}

\begin{proof} [Proof of Theorem \ref{th:existencemmap}]
By hypothesis, we can assume that $\omega$ has constant scalar curvature and that $h$ is a solution of the vortex equation \eqref{eq:vortexeq}. By Lemma \ref{lem:deltaBcscK}, the linearization $\delta B^0$ is injective, since $g(\Sigma) \geqslant 1$ implies that there are no non-zero Hamiltonian Killing vector fields on $(\Sigma,\omega)$. Thus, positivity of $\Delta_\omega +  |\phi|_h^2$ implies that $\delta B^0$ is an isomorphism. The result follows by the Implicit Function Theorem on Banach spaces, by taking suitable Sobolev completions of $C^\infty(\Sigma)$. Regularity of the solutions follows easily from a boot-strapping argument, using regularity for the Laplacian and the constant scalar curvature operator (cf. \cite[Lemma 4.3]{AGG}).
\end{proof}





We note that the previous result cannot be obtained using the methods of \cite{AGG}, due to the existence of non-trivial Hamiltonian Killing  vector fields on $\Sigma \times \PP^1$.

\begin{remark}
For $g(\Sigma) \geqslant 1$, adapting the argument in \cite[Section 4.4]{AGG} one can easily prove that the existence of solutions of the gravitating vortex equations is an open condition for $(J,A,\phi) \in \cT$, and also for the parameters $\alpha$, $\tau$ and $\Vol(\Sigma)$.
\end{remark}

\section{Yang's theorem and Geometric Invariant Theory}\label{sec:cosmicstring}

We fix a constant $\alpha> 0$ and a symmetry breaking parameter $\tau > 0$. In
this section we focus on a particular case of the gravitating vortex
equations, given by $c_1(L) > 0$ and the condition $c=0$ in \eqref{eq:gravvortexeq1}, that is,
\begin{equation}\label{eq:cosmicstrings}
\begin{split}
i\Lambda_\omega F_h + \frac{1}{2}(|\phi|_h^2-\tau) & = 0,\\
S_\omega + \alpha(\Delta_\omega + \tau) (|\phi|_h^2 -\tau) & = 0.
\end{split}
\end{equation}

Following Yang \cite{Yang1992,Yang1994CMP}, we shall refer to \eqref{eq:cosmicstrings} as the \emph{Einstein--Bogomol'nyi equations}. Note that $c = 0$ combined with $c_1(L) > 0$ constrains the topology of the Riemann surface to be $\Sigma = \PP^1$, as this condition is equivalent to
$$
\chi(\Sigma) = 2\alpha\tau c_1(L),
$$
which is positive for non-zero $\phi$ (see Theorem \ref{th:B-GP}).

\subsection{Existence of solutions and Yang's Theorem}


The particular features of the Einstein--Bogomol'nyi 
equations \eqref{eq:cosmicstrings} are better observed using the
Kazdan--Warner  type formulation of the equations \eqref{eq:KWtype}, 
as in this case the system reduces to a single partial differential equation \cite{Yang}
\begin{equation}\label{eq:single}
\begin{split}
\Delta f + \frac{1}{2}e^{2u}(e^{2f}|\phi|^2-\tau) & = - N,
\end{split}
\end{equation}
for a function $f \in C^\infty(\PP^1)$, where 
$$
u = 2 \alpha \tau f - \alpha e^{2f}|\phi|^2 + c'.
$$ 
Here, $\Delta$ is the Laplacian of the Fubini--Study metric on $\PP^1$, normalized so that $\int_{\PP^1}\omega_{FS} = 2\pi$, $|\phi|$ is the norm with respect to the Fubini--Study metric on $L = \mathcal{O}_{\PP^1}(N)$ and $c'$ is a real constant that can be chosen at will. 

The Liouville type equation \eqref{eq:single} on $\PP^1$ was studied by Yang \cite{Yang,Yang3}, who proved the following existence result. Let 
$$
D = \sum_j n_j p_j
$$ 
be the effective divisor on $\PP^1$ corresponding to a pair $(L,\phi)$, with $N = \sum_j n_j = c_1(L)$.

\begin{theorem}[Yang's Existence Theorem]\label{th:Yang}
Assume that \eqref{eq:ineq} holds. Then, there exists a solution of the Einstein--Bogomol'nyi equations \eqref{eq:cosmicstrings} on $(\PP^1,L,\phi)$ if one of the following conditions hold
\begin{enumerate}

\item[\textup{(1)}] $D = \frac{N}{2}p_1 + \frac{N}{2}p_2$, where $p_1 \neq p_2$ and $N$ is even. In this case the solution admits an $S^1$-symmetry. 

\item[\textup{(2)}] $n_j < \frac{N}{2}$ for all $j$.

\end{enumerate}
\end{theorem}

The conditions in part (2) of Theorem \ref{th:Yang} appear in \cite{Yang} in
an a priori estimate for a weak solution of the equation, obtained by
smoothing $\log |\phi|$ around the zeros of $\phi$ and taking a suitable
limit. The solution satisfying
part (1) of Theorem \ref{th:Yang} is constructed in \cite{Yang3} assuming an $S^1$-invariant ansatz for the solution and solving an ordinary differential equation. We note that this last part is stated in the original result for $D = \frac{N}{2}p + \frac{N}{2}\overline{p}$, with $p,\overline{p}$ antipodal points on $\PP^1$ and $S^1$ symmetry given by rotation along the $\{p,\overline{p}\}$ axis. The more general situation stated here reduces to Yang's result, after pull-back of the solution by an element in $\SL(2,\CC)$ which takes $p_1,p_2$ to a pair of antipodal points.



A remarkable fact is that the conditions that appear in Theorem \ref{th:Yang} 
have a natural 
meaning in the construction  of quotients of algebraic varieties  by complex reductive Lie groups, 
given by  Geometric Invariant Theory (GIT) \cite{MFK,ThomasGIT}. Consider the natural $\operatorname{SL}(2,\CC)$-action on $\PP^1$. There exists a unique lift of this action to $L$, which induces an action on $H^0(L)$. GIT tells us that, to construct the algebraic quotient of $\PP(H^0(L))$ by $\SL(2,\CC)$, we need to distinguish an invariant dense open subset 
$$
(\PP(H^0(L)))^{ss} \subset \PP(H^0(L))
$$ 
given by \emph{semi-stable orbits} on $H^0(L)$, in order to avoid non-Hausdorff phenomena caused by the non-properness of the group action. Note that $\PP(H^0(L))$ can be identified naturally with $S^N\PP^1$, the space of length $N$ divisors on $\PP^1$. Upon restriction to $(\PP(H^0(L)))^{ss}$ and identification of a semi-stable orbit with the unique closed orbit on its closure, the topological quotient
$$
\PP(H^0(L))/\!\!/\SL(2,\CC),
$$
inherits a structure of algebraic variety. Closed orbits on $H^0(L)$ are called \emph{polystable}, and they are in correspondence with points in the GIT quotient, while orbits excluded from the GIT quotient, those in $\PP(H^0(L)) \backslash (\PP(H^0(L)))^{ss}$, are called \emph{unstable}. Among polystable orbits, those with finite isotropy group are called \emph{stable}. By the \emph{Hilbert--Mumford criterion}, stable, polystable and unstable orbits are conveniently characterised in terms of a numerical criterion, related with conditions $(1)$ and $(2)$ of Theorem \ref{th:Yang}.


\begin{proposition}[\cite{MFK}]\label{prop:GIT}
The orbit of $\phi \in H^0(L)$ is 
\begin{enumerate}
\item[\textup{(1)}] strictly polystable if and only if $(1)$ in Theorem \ref{th:Yang} is satisfied,

\item[\textup{(2)}] stable if and only if $(2)$ in Theorem \ref{th:Yang} is satisfied,

\item[\textup{(3)}] unstable if and only if there exists $p_j \in D$ such that $n_j > \frac{N}{2}$. 
\end{enumerate}
\end{proposition}



The existence of solutions of the Einstein--Bogomol'nyi equations on $(\PP^1,L,\phi)$ turns out to be equivalent to the polystability of $\phi$, as we prove in \cite{AGG2}, where we will also address the solution of a conjecture by Yisong Yang \cite{Yang3} (see also \cite[p. 437]{Yangbook}) about the non-existence of Nielsen--Olesen strings superimposed at a single point (corresponding to the unstable configuration $D = Np$ in Proposition \ref{prop:GIT}).



\begin{remark}
Theorem \ref{th:Yang} combined with an Implicit Function Theorem argument --- along the lines of the proof of \cite[Theorem 4.10]{AGG} --- imply an existence result for the gravitating vortex equations \eqref{eq:gravvortexeq1} on $\PP^1$ for $0 < |c| \ll 1$. For this, we note that if $\phi \in H^0(L)$ is stable then $\Aut (\PP^1,L,\phi)$, the group automorphisms of the total space of $L$ preserving $\phi$ (see \cite{AGG2}), is finite. The method in Section \ref{sec:weakbis} implies that the linearization of the Einstein--Bogomol'nyi equations is invertible, and hence Yang's solution with $c = 0$ can be deformed to solutions of the gravitating vortex equation with $\alpha$ nearby $\tau^{-1}N^{-1}$.
\end{remark}

\subsection{Deformation of symmetric gravitating vortices via GIT}\label{sec:qualitative}

In this section we prove an existence result for gravitating vortices around a symmetric solution of \eqref{eq:gravvortexeq1}. For this, we will exploit the relation of the gravitating vortex equations with the K\"ahler--Yang--Mills equations by dimensional reduction, obtained in Proposition \ref{prop:dimred}, and apply the main result in \cite{GFT} (which generalizes a theorem of Sz\'ekelyhidi in \cite{Szekelyhidi}). In particular, we give a GIT point of view of Yang's Existence Theorem \ref{th:Yang} around a symmetric solution.

Let $\phi$ be a holomorphic section of $L$ on $\PP^1$ which vanishes at exactly two points with multiplicity $N/2$. For a choice of homogeneous coordinates $[x_0,x_1]$, we can assume that $\phi$ vanishes at the two antipodal points $0, \infty$, and therefore it can be identified with the 
homogeneous polynomial
\begin{equation}\label{eq:phisymmetric}
\phi \cong t x_0^{N/2}x_1^{N/2}
\end{equation}
for a suitable $t \in \CC^*$. Without loss of generality we take $t =1$.

Consider the $\operatorname{SL}(2,\CC)$-action on $\PP^1$. There exists a unique lift of this action to $L$, which induces an action on $H^0(L)$. Upon identification of $H^0(L)$ with the space of degree $N$ homogeneous polynomials in $x_0,x_1$, the $\operatorname{SL}(2,\CC)$-action is simply given by pull-back, and the isotropy group of $\phi$ for this action is $\CC^*$, regarded as the subgroup of $\operatorname{SL}(2,\CC)$

\begin{equation}\label{eq:S1action}
\begin{split}
\( \begin{array}{cc}
\lambda & 0 \\
0 & \lambda^{-1}
\end{array} \) \subset \operatorname{SL}(2,\CC).
\end{split}
\end{equation}

Denote by $s_0,s_1,s_2$ the three standard generators of $\mathfrak{sl}(2,\CC)$, with relations
$$
[s_0,s_1] = s_2, \qquad [s_2,s_0] = 2s_0, \qquad [s_2,s_1] = -2s_1,
$$
where $s_0$ denotes the generator of the Abelian Lie algebra $\CC \subset \mathfrak{sl}(2,\CC)$ of \eqref{eq:S1action}. Consider the natural exact sequence
\begin{equation}\label{eq:exactseq2}
  \xymatrix{
    0 \ar[r] & \CC \ar[r] & \mathfrak{sl}(2,\CC) \ar[r]^{\rho} & H^0(L) \ar[r]^{\iota \qquad \qquad} & H^0(L)/\langle \rho(s_1),\rho(s_2) \rangle \ar[r] & 0
  }
\end{equation}
where $\rho$ denotes the infinitesimal action of $\operatorname{SL}(2,\CC)$ on $\phi \in H^0(L)$ and $\iota$ is the projection. Note that the $\CC^*$-action preserves $\langle \rho(s_1),\rho(s_2) \rangle$ and hence $\iota$ is $\CC^*$-equivariant. Using the natural basis of $H^0(L)$ given by the monomials of degree $N$, we can identify the representation $H^0(L)/\langle \rho(s_1),\rho(s_2) \rangle $ with the complement of $\langle \rho(s_1),\rho(s_2) \rangle$ on $H^0(L)$
$$
\langle \rho(s_1),\rho(s_2) \rangle^\perp =  \{\phi'' = \sum_{j=0}^N a_j x_0^j x_1^{N-j}: a_{\frac{N}{2}-1} = a_{\frac{N}{2}+1} = 0\},
$$
where we use that
$$
\rho(s_1) = \frac{N}{2}x_0^{\frac{N}{2}-1}x_1^{\frac{N}{2}+1}, \qquad \rho(s_2) = \frac{N}{2}x_0^{\frac{N}{2}+1}x_1^{\frac{N}{2}-1}.
$$
We are ready to state the main result of this section. Recall that an element $\phi' \in H^0(L)$ is polystable if and only if its $\operatorname{SL}(2,\CC)$-orbit on $H^0(L)$ is closed.

\begin{theorem}\label{th:triumph}
Assume that $\phi$ is as in \eqref{eq:phisymmetric}, and that it admits a solution of the gravitating vortex equations \eqref{eq:gravvortexeq1} on $\PP^1$ fixed by $S^1 \subset \CC^*$. Then, for any $\phi'' \in \langle \rho(s_1),\rho(s_2) \rangle^\perp$ close enough to $\phi$ the following holds:
\begin{enumerate}

\item[\textup{(1)}] $\phi''$ is polystable if and only if the $\CC^*$-orbit of $\phi''$ is closed,

\item[\textup{(2)}] if $\phi''$ is polystable then it admits a solution of the gravitating vortex equations.

\end{enumerate}
\end{theorem}

Note that for the case of the Einstein--Bogomol'nyi equations an $S^1$-symmetric solution exists by Theorem \ref{th:Yang} $(1)$. We expect that the methods of \cite{Yang3} can be adapted to prove the existence of $S^1$-symmetric gravitating vortices on $\PP^1$ for $c \neq 0$. 
Theorem \ref{th:triumph} $(1)$ provides a characterization of polystable points $\phi'' \in \langle \rho(s_1),\rho(s_2) \rangle^\perp$, in terms of a particular $1$-parameter subgroup (see Lemma \ref{lem:GIT}). Theorem \ref{th:triumph} $(2)$ ensures the existence of solutions for any polystable $\phi''$ close enough $\phi$, in the complement of the $\mathfrak{sl}(2,\CC)$-action, recovering a qualitative version of Yang's Theorem around the symmetric solution.

To proceed with the proof, assume for a moment that $\phi$ is an arbitrary section of $H^0(L)$. To avoid confusion, we introduce the notation $\tilde \PP^1$ for a different copy of the Riemann sphere, and consider the holomorphic $\widetilde{SU}(2)$-equivariant bundle $E$ on $X = \PP^1 \times \tilde \PP^1$ determined by $\phi$, defined in \eqref{eq:bundleE}. Here, $\widetilde{SU}(2)$ acts on $\tilde \PP^1 \cong \widetilde{SU}(2)/\widetilde{U}(1)$ in the standard way. Assume further that $(\PP^1,L,\phi)$ admits a solution $(\omega,h)$ of the gravitating vortex equations.  By Proposition \ref{prop:dimred}, $(\omega,h)$ determines a $\widetilde{SU}(2)$-invariant solution of the K\"ahler--Yang--Mills equations $(\omega_\tau,H)$ on $(X,E)$. Let $\cX_\tau$ the extended gauge group of $(\omega_\tau,H)$, and define
$$
K = \cX_\tau \cap \Aut (X,E),
$$
where $\Aut (X,E)$ denotes the space of holomorphic automorphisms of $E$ covering an automorphisms of $X$. The group $K$ is compact and finite-dimensional, and hence admits a complexification $K^c$. In \cite{GFT}, a finite-dimensional representation
$$
K^c \curvearrowright H^1(X,L^*_{\omega_\tau})
$$
parameterizing infinitesimal deformations of the pair $(X,E)$ compatible with $\omega_\tau$ 
is constructed, using elliptic operator theory. The compatibility condition amounts to considering infinitesimal deformations of $(X,E)$ such that the corresponding infinitesimal deformation of $X$ is compatible with the symplectic structure $\omega_\tau$. By application of the Kuranishi method, any small deformation of $(X,E)$ determines a point in $H^1(X,L^*_{\omega_\tau})$. Theorem \ref{th:triumph} will follow from the following result.

\begin{theorem}[\cite{GFT}]\label{th:GFT}
Any small deformation of $(X,E)$ with closed $K^c$-orbit in $H^1(X,L^*_{\omega_\tau})$ admits a solution of the K\"ahler--Yang--Mills equations.
\end{theorem}

\begin{remark}
The previous theorem relies on the moment map interpretation of the K\"ahler--Yang--Mills equations, and follows adapting arguments for the constant scalar curvature K\"ahler equation due to Sz\'ekelyhidi \cite{Szekelyhidi} (see also \cite{Tipler}). 
\end{remark}

In the particular case we are dealing with, $h^{0,1}(X) = 0 = h^{0,2}(X)$ and $X$ is simply connected, which implies the existence of an exact sequence (see \cite[Eq. (34)]{GFT})
\begin{equation}\label{eq:exactseq}
  \xymatrix{
    0 \ar[r] & H^0(\operatorname{End} E) \ar[r] & \Lie \operatorname{Aut} (X,E) \ar[r] & \Lie \operatorname{Aut} X \\
 \ar[r]^{\rho \qquad} & H^1(\operatorname{End} E)   \ar[r]^{\iota} & H^1(X,L^*_{\omega_\tau}) \ar[r] & 0
  }
\end{equation}
such that the map $\iota$ is $K^c$-equivariant with respect to the $K^c$-action on $H^1(\operatorname{End} E)$ by pull-back, and $\rho$ is induced by pull-back of $E$ with respect to elements in $\operatorname{Aut} X$. 

To prove Theorem \ref{th:triumph}, we need an explicit description of the group $K^c$ and to relate \eqref{eq:exactseq} with the exact sequence \eqref{eq:exactseq2}, when $\phi$ is an in \eqref{eq:phisymmetric}.

\begin{lemma}\label{lem:Kc}
Suppose that the solution $(\omega,h)$ is invariant by the action of $S^1 \subset \CC^*$ (see \eqref{eq:S1action}). Then, 
$$
K = S^1/\mathbb{Z}_2 \times \widetilde{PU}(2) \times \mathring{S}^1, 
$$
where $\widetilde{PU}(2)$ acts on $\tilde \PP^1$ and $\mathring{S}^1$ is another copy of $S^1$, which acts by multiplication on the fibres of $E$. Furthermore,
$$
\Aut (X,E) = K^c = \CC^*/\mathbb{Z}_2 \times \widetilde{\operatorname{PGL}}(2,\CC) \times \mathring{\CC}^*.
$$
\end{lemma}

\begin{proof}
We claim that the group of isometries $\operatorname{Isom}(\omega_\tau)$ of $\omega_\tau = \omega + \frac{4}{\tau}\tilde{\omega}_{FS}$ is given by 
\begin{equation}\label{eq:isom}
\operatorname{Isom}(\omega_\tau) = S^1/\mathbb{Z}_2 \times \widetilde{PU}(2).
\end{equation}
By hypothesis on $\omega$ we have $S^1/\mathbb{Z}_2 \times \widetilde{PU}(2) \subset \operatorname{Isom}(\omega_\tau)$. For the other inclusion, we note that $S^1/\mathbb{Z}_2$ is contained in a maximal compact subgroup of $\operatorname{PGL}(2,\CC) = \Aut(\PP^1)$, since $\omega$ is K\"ahler. Maximal compact subgroups of $\operatorname{PGL}(2,\CC)$ are isomorphic to $\PU(2)$, by to conjugation, and for $S^1/\mathbb{Z}_2 \subset G \subset \PU(2)$ a chain of compact subgroups, one has $G = S^1/\mathbb{Z}_2$ or $G = \PU(2)$. We are left with only two possibilities. If the isometry group is $\PU(2) \times \widetilde{PU}(2)$ we have $\omega = t \omega_{FS}$ for a suitable choice of $t > 0$. By the choice of $\phi$ \eqref{eq:phisymmetric}, a direct check shows that in this case $(\omega,h)$ is not a solution of the gravitating vortex equations (with this ansatz, some manipulations on \eqref{eq:gravvortexeq1} imply $2N = \tau$, which contradicts Theorem \ref{th:B-GP}). Hence, necessarily \eqref{eq:isom} holds. 

Since $X$ is simply connected, isometries for $\omega_\tau$ are always Hamiltonian and we obtain a monomorphism 
$$
\theta \colon S^1/\mathbb{Z}_2 \times \widetilde{PU}(2) \to K.
$$
This monomorphism is constructed identifying $E \cong p^*L \oplus q^*\mathcal{O}_{\tilde{\PP}^1}(2)$ as a smooth complex vector bundle, and using the natural lift of the $\PU(2) \times \widetilde{PU}(2)$-action on $X$ to $p^*L \oplus q^*\mathcal{O}_{\tilde{\PP}^1}(2)$ ($N = c_1(L)$ is even by assumption). Note that the $S^1/\mathbb{Z}_2$-action on $\PP^1$ lifts to $L$, preserving the section $\phi$, and hence the induced $S^1/\mathbb{Z}_2 \times \widetilde{PU}(2)$-action on $E$ is holomorphic. The first part of the statement follows easily, from the fact the automorphisms of $E$ which project to the identity automorphism on $X$ are necessarily constant.

As for the second part, the proof follows from a direct calculation using the explicit form of the extension \eqref{eq:bundleE} or, alternatively, by applying a Matsushima--Lichnerowicz type theorem for the K\"ahler--Yang--Mills equations \cite{AGG2}.
\end{proof}

To relate \eqref{eq:exactseq} with the exact sequence \eqref{eq:exactseq2}, we need the following fact.
\begin{lemma}\label{lem:inclusion}
There is an injection
\begin{equation}\label{eq:inclusion}
H^0(L) \hookrightarrow H^1(\operatorname{End} E)^{\widetilde{\operatorname{PGL}}(2,\CC)} \subset H^1(\operatorname{End} E)
\end{equation}
where $H^1(\operatorname{End} E)^{\widetilde{\operatorname{PGL}}(2,\CC)}$ denotes the space of $\widetilde{\operatorname{PGL}}(2,\CC)$-invariant infinitesimal deformations of $E$. 
\end{lemma}
\begin{proof}
This map is constructed associating to $\phi' \in H^0(L)$ the $\widetilde{\operatorname{PGL}}(2,\CC)$-equivariant extension \eqref{eq:bundleE} corresponding to $\phi + \phi'$. More explicitely, it is given by
$$
\phi' \mapsto [\beta(\phi')],
$$
where $[\beta(\phi')]$ denotes the harmonic part of
$$
\beta(\phi') = p^*\phi' \otimes q^*\eta \in \Omega^{0,1}(X,\Hom(q^*\cO_{\tilde \PP^1}(2),p^*L))
$$
with respect to $(\omega_\tau,H)$, and 
$$
\eta = \frac{\sqrt{8\pi}}{\tau}\frac{dz \otimes d\overline{z}}{(1 + |z|^2)^2} \in \Omega^{0,1}(\tilde \PP^1,\Hom(\cO_{\tilde \PP^1}(2),\cO_{\tilde \PP^1}))
$$
for a choice of coordinate $z = \frac{\tilde x_0}{\tilde x_1}$ on $\tilde \PP^1$. 
\end{proof}

To apply Theorem \ref{th:GFT}, it is convenient to do the following change of variables in \eqref{eq:inclusion}: since $\phi$ is fixed by $\CC^*$ (see \eqref{eq:S1action}), the following bijective map is $\CC^*$-equivariant
$$
H^0(L) \to H^0(L)\colon \phi' \mapsto \phi + \phi'.
$$
Therefore, for a small neighbourhood 
$$
\phi \in U \subset H^0(L)
$$ 
we can assume that Theorem \ref{th:GFT} applies to $\iota(\phi' - \phi)$ for any element $\phi ' \in U$, where we abuse of the notation and identify $\phi' - \phi$ with its image in $H^1(\operatorname{End} E)$. 

From Lemma \ref{lem:Kc} and Lemma \ref{lem:inclusion}, it is now easy to see that the sequence \eqref{eq:exactseq} restricts to the exact sequence \eqref{eq:exactseq2}, where 
$$
\iota (H^0(L)) \subset H^1(X,L_{\omega_\tau}^*)
$$
is identified with $\iota(H^0(L)) \cong H^0(L)/\langle \rho(s_1),\rho(s_2) \rangle$. Note that $\iota (H^0(L))$ inherits a linear action of 
$$
\CC^*/\mathbb{Z}_2 = K^c/(\widetilde{\operatorname{PGL}}(2,\CC) \times \mathring{\CC}^*)
$$ 
such that the previous identification is equivariant. Since the action of $\mathbb{Z}_2 \subset  \CC^*$ on $H^0(L)$ is trivial, we will conveniently work with $\CC^*$ instead.

For the proof of Theorem \ref{th:triumph} (1) we need the following lemma.

\begin{lemma}\label{lem:GIT}
For any $\phi'' \in \langle \rho(s_1),\rho(s_2) \rangle^\perp$ such that $\phi'' \neq \phi$, the $\CC^*$-orbit of $\phi''$ is closed if and only if $\phi''$ vanishes at $0$ and $\infty$ with multiplicity $< N/2$.
\end{lemma}
\begin{proof}
Given $\phi'' \neq \phi$, as $\lambda \in \CC^*$ tends to $0$, the first $N/2-1$ monomials in the expression for $\phi''$ goes to infinity. Thus $\lambda \cdot \phi''$ tends to $\infty$ and the orbit is closed about $\lambda \to 0$ unless $a_j = 0$ for $j \leq N/2 -2$. That is, it is closed about $0$ so long as $\phi''$ does not vanish to order $\geq N/2$ at $x_0 = 0$. The statement follows arguing now for $\lambda \to \infty$.
\end{proof}

We are ready to prove the main theorem of this section.

\begin{proof}[Proof of Theorem \ref{th:triumph}]
We first prove $(2)$. By definition of polystability, the $\operatorname{SL}(2,\CC)$-orbit of $\phi''$ on $H^0(L)$ is closed and, in particular, the $\CC^*$-orbit of $\phi''$ in $\langle \rho(s_1),\rho(s_2) \rangle^\perp$ is closed. Then, by Theorem \ref{th:GFT}, the small deformation $(X,E')$ of $(X,E)$ determined by $\phi'$ admits a solution of the K\"ahler--Yang--Mils equations. We can now average to produce a $\widetilde{SU}(2)$-invariant solution, which implies the existence of a gravitating vortex by Proposition \ref{prop:dimred}.

As for the proof of $(1)$, it suffices to show that if $\phi''$ satisfies the hypothesis of Lemma \ref{lem:GIT} then it is polystable. Let $\phi' \in H^0(L)$ close to $\phi$. Since $\phi$ has only zeros at $0$ and $\infty$ with multiplicity $N/2$, $\phi'$ is forced to be semistable. We will prove that if $\phi'$ is strictly semi-stable and vanishes at $0$ and $\infty$ with multiplicity $< N/2$, then $\phi' \notin \langle \rho(s_1),\rho(s_2) \rangle^\perp$. We do the proof for the case that $\phi'$ vanishes at $3$ points $p_0$, $p_1$, $p_2$, with multiplicities $N/2$, $k$, $l$, and leave the general case for the reader. We assume $p_0$ close to $0$ and $p_1,p_2$ close to $\infty$. By assumption, $p_0 \neq 0$ and $p_1 \neq \infty \neq p_2$, so that but we can assume
$$
\phi' = (z - z_0)^{N/2}(\epsilon z - z_1)^k (\epsilon z - z_2)^l
$$
in coordinate $z = x_0/x_1$, with $z_0,z_1,z_2 \neq 0$ and $0 < \epsilon \ll 1$. A direct calculation then shows that
$$
\phi' = \sum_{j= 0}^{N} z^j \sum_{\tiny{\begin{array}{l}
j_0 + j_1 + j_2 = j \\ 0 \leq j_0 \leq N/2 \\ 0 \leq j_1 \leq k \\ 0 \leq j_2 \leq l
\end{array}}}^{N} (-1)^{N/2 - j_1 - j_2} \epsilon^{j_1+j_2} \binom{N/2}{j_0}\binom{k}{j_1}\binom{l}{j_2}z_0^{N/2 - j_0}z_1^{k-j_1}z_2^{l-j_2}
$$
and writting $\phi' = \sum_{j=0}^N a_j z^j$ we have
$$
a_{N/2 -1} = - \frac{N}{2}z_0 z_1^k z_2^l + O(\epsilon). 
$$
Therefore $a_{N/2 -1} \neq 0$ and we conclude that $\phi' \notin \langle \rho(s_1),\rho(s_2) \rangle^\perp$, as claimed.
\end{proof}

The proof of Theorem \ref{th:triumph} $(1)$ follows ultimately by convexity properties of the coefficient $a_{N/2 -1}$ along a $\operatorname{SL}(2,\CC)$-orbit in $H^0(L)$. For this, we note that for $\phi'$ strictly semistable there exists a $1$-parameter subgroup $\lambda \colon \CC^* \to  \operatorname{SL}(2,\CC)$ such that
$$
\lim_{\lambda \to 0} \lambda \cdot \phi' = \phi.
$$
A more conceptual proof can be derived, indeed, from Sz\'ekelyhidi's
arguments in \cite{Szekelyhidi}. Considering the finite-dimensional
moment map $\mu$ for the $\SU(2)$-action on $\PP(H^0(L))$, identified
with $S^N(\PP^1)$, we have that $\mu(\phi) = 0$ (see
\cite{ThomasGIT}). Then, applying Sz\'ekelyhidi's argument to $\mu$,
we obtain that any $\phi'' \in \langle \rho(s_1),\rho(s_2) \rangle^\perp$
with closed $\CC^*$-orbit admits also a zero of $\mu$ on its
$\operatorname{SL}(2,\CC)$-orbit. Therefore, by the Kempf--Ness Theorem any
such $\phi''$ is polystable.

\end{document}